\documentclass{amsart}

\usepackage{tikz}
\tikzset{
vertex/.style={
  circle, draw=black, very thick, minimum size=0.15cm, fill=black,
  inner sep=0, outer sep=0},
 }

\usepackage{color}

\usepackage[hang,small,bf]{caption}
\definecolor{myurlcolor}{rgb}{0.6,0,0}
\definecolor{mycitecolor}{rgb}{0,0,0.8}
\definecolor{myrefcolor}{rgb}{0,0,0.8}
\usepackage[pagebackref]{hyperref}
\hypersetup{colorlinks, linkcolor=myrefcolor, citecolor=mycitecolor, urlcolor=myurlcolor}

\usepackage{mathtools}
\DeclarePairedDelimiter{\floor}{\lfloor}{\rfloor}

\newcommand{\maps}{\colon} 
\newcommand{\cutwidth}{\mathrm{cw}}
\newcommand{\wirelength}{\mathrm{wl}}
\newcommand{\bandwidth}{\mathrm{bw}}
\newcommand{\type}{\mathrm{Type}}
\newcommand{\Split}{\mathrm{Split}}
\newcommand{\ccw}{\mathrm{ccw}}
\newcommand{\lcw}{\mathrm{lcw}}
\newcommand{\cwl}{\mathrm{cwl}}
\newcommand{\lwl}{\mathrm{lwl}}
\newcommand{\cbw}{\mathrm{cbw}}
\newcommand{\lbw}{\mathrm{lbw}}
\newcommand{\wl}{\mathrm{wl}}
\newcommand{\bw}{\mathrm{bw}}
\newcommand{\cw}{\mathrm{cw}}

\newcommand{\define}[1]{{\bf \boldmath{#1}}}
\newtheorem{thm}{Theorem}

\newtheorem{lemma}[thm]{Lemma}

\theoremstyle{definition}
\newtheorem{case}{Case}
\newtheorem{claim*}{Claim}

\title{The Cyclic Cutwidth of $Q_n$}
\author{Jason Erbele\\with\\Dr.\ Joseph Chavez \& Dr.\ Rolland Trapp}
\date{August 21, 2003}
\email{erbele@math.ucr.edu}

\begin{document}

\begin{abstract}
In this article the cyclic cutwidth of the \(n\)-dimensional cube is explored.  It has been
conjectured by Dr.\ Chavez and Dr.\ Trapp that the cyclic cutwidth of \(Q_n\) is minimized with the
Graycode numbering.  Several results have been found toward the proof of this conjecture.
\end{abstract}

\maketitle

\section{Introduction}
\label{intro}
Let \(G = (V, E, \partial)\) represent a graph with a set, \(V\), of vertices, a set, \(E\), of edges, and
a function \(\displaystyle \partial \maps E \to \binom{V}{2}\) which identifies the two distinct
vertices incident to each edge.  \(G\) has often been analogised to an electric circuit in the
literature.

A numbering of the vertices, \(\eta\), is a function that assigns a distinct number from \(1\) to
\(m\) to each of the vertices in \(G\), where \(m = |V|\).  A numbering can most naturally be
thought of as an embedding of \(G\) into a linear chassis, though other host graphs may be
considered.  The main emphasis of this paper will be with a circular host graph.  To distinguish
these two host graphs, the letters `\emph{l}' and `\emph{c}' will be used as prefixes for linear and
cyclic, respectively.

There are three major properties of an embedding of a graph: \emph{bandwidth} (\(\bandwidth\)),
\emph{wirelength} (\(\wirelength\)), and \emph{cutwidth} (\(\cutwidth\)) \cite{Chung}.
\[   \lbw(G,\eta) = \max\left\{|\eta(v) - \eta(w)| : \{v,w\} \in E\right\}.\]

\noindent
That is, \define{\(\lbw\)} is the maximum distance between two vertices connected by an edge.  For the
graph, \(\lbw(G)\) is the minimum of these \(\lbw(G,\eta)\)'s over all numberings.
\[   \lwl(G,\eta) = \sum_{\{v,w\} \in E} |\eta(v) - \eta(w)|.\]

\noindent
That is, \define{\(\lwl\)} is the sum of the lengths of all the edges.  \(lwl(G)\) is the minimum,
again, over all numberings.
\[   \lcw(G,\eta) = \max_\ell |\{\{v,w\} \in E : \eta(v) \leq \ell < \eta(w)\}|.\]

\noindent
That is, \define{\(\lcw\)} is the maximum number of edges that pass between two consecutively
numbered vertices. \(\lcw(G)\) is the minimum of these maxima over all numberings.  \(\cbw,\)
\(\cwl,\) and \(\ccw\) are defined similarly to their linear counterparts, with the appropriate
adjustments made, particularly, vertices are numbered congruence classes instead of numbers.  In
this paper only \(\wl\) and \(\cw\) are of interest.  \textbf{Note:}  For \(\cbw\), \(\cwl\), and
\(\ccw\) there are two choices for which direction an edge should go.  For \(\cbw\) and \(\cwl\) we
clearly only want to choose the direction that minimizes the length of the edge.  However, there are
graphs with numberings that have a smaller \(\ccw\) when an edge goes the long way around.

In finding the values of \(\cw(G)\), \(\wl(G)\) and \(\bw(G)\), a useful function from the area of
discrete isoperimetric problems, the theta function, can be used.  This function will be limited in
its use here as follows:
\[   \theta(S) = |\{v,w\} \in E,\ v \in S,\ w \notin S|\]
\[   \text{and } \theta(\ell) = \min_{|S| = \ell} \theta(S).\]

\noindent
In other words, \(\theta(S)\), \(S \subseteq V\) is the number of edges that have exactly one vertex
in \(S\), and over all sets \(S \subseteq V\) of size \(\ell\), \(\theta(\ell)\) is the least number
of edges that have exactly one vertex in the set.  The notation \(\theta_n(S)\) and
\(\theta_n(\ell)\) will be used when \(G\) is \(Q_n\).  For clarity, lowercase letters will be used
to represent numbers and uppercase letters will be reserved for sets.

Finding the value of \(\cw(G)\)  is called \emph{the cutwidth problem}.  The cutwidth problem is
NP-complete for graphs in general \cite{Chung}.  However, the solution to the cutwidth problem is
known for special cases like an \(n\)-dimensional cube (\(Q_n\)) embedded on linear and grid host
graphs.  (When the host graph is a grid, the term \emph{congestion} is used instead of cutwidth.)  A
conjecture has been made for \(\cw(Q_n)\) when the host graph is a circle, called the CT conjecture.

The CT conjecture (named after Chavez and Trapp) states that the Graycode numbering gives
\(\ccw(Q_n)\).  Or, as a formula, the CT conjecture asserts
\(\ccw(Q_n) = \floor*{ \frac{5 \cdot 2^{n-2}}{3} }\)
when \(n \geq 2\).  The Graycode numbering is recursively defined with a base case in \(Q_2\).
\(Q_n\) is two copies of \(Q_{n-1}\), so if we know the numbering for \(Q_{n-1}\), that numbering is
copied in reverse on the second \(Q_{n-1}\) to give the numbering for \(Q_n\).  One feature of the
Graycode numbering is that consecutively numbered vertices are adjacent to each other.

  \begin{figure}[!h]
    \centering
      \begin{tikzpicture}
        \draw (0,0) -- +(2,0);
        \draw (0,2) -- +(2,0);
        \draw (0.5,0.5) -- +(2,0);
        \draw (0.5,2.5) -- +(2,0);
        \draw (0,0) -- +(0,2);
        \draw (2,0) -- +(0,2);
        \draw (0.5,0.5) -- +(0,2);
        \draw (2.5,0.5) -- +(0,2);

        \draw[dashed] (0,0) -- +(0.5,0.5);
        \draw[dashed] (2,0) -- +(0.5,0.5);
        \draw[dashed] (0,2) -- +(0.5,0.5);
        \draw[dashed] (2,2) -- +(0.5,0.5);
        \node at (0.1,-0.2) {1};
        \node at (0.7,0.7) {1$'$};
        \node at (0.2,1.8) {2};
        \node at (0.7,2.7) {2$'$};
        \node at (1.8,1.8) {3};
        \node at (2.5,2.7) {3$'$};
        \node at (2,-0.2) {4};
        \node at (2.7,0.7) {4$'$};
        \draw (4,0) -- +(2,0);
        \draw (4,2) -- +(2,0);
        \draw (4.5,0.5) -- +(2,0);
        \draw (4.5,2.5) -- +(2,0);
        \draw (4,0) -- +(0,2);
        \draw (6,0) -- +(0,2);
        \draw (4.5,0.5) -- +(0,2);
        \draw (6.5,0.5) -- +(0,2);

        \draw (4,0) -- +(0.5,0.5);
        \draw (6,0) -- +(0.5,0.5);
        \draw (4,2) -- +(0.5,0.5);
        \draw (6,2) -- +(0.5,0.5);
        \node at (4.1,-0.2) {1};
        \node at (4.2,1.8) {2};
        \node at (5.8,1.8) {3};
        \node at (6,-0.2) {4};
        \node at (6.7,0.7) {5};
        \node at (6.5,2.7) {6};
        \node at (4.7,2.7) {7};
        \node at (4.7,0.7) {8};
      \end{tikzpicture}
    \caption[Generating the Graycode numbering for \(Q_3\) from \(Q_2\)]{Generating the Graycode
                                            numbering for \(Q_3\) from \(Q_2\).\label{graycode}}
  \end{figure}
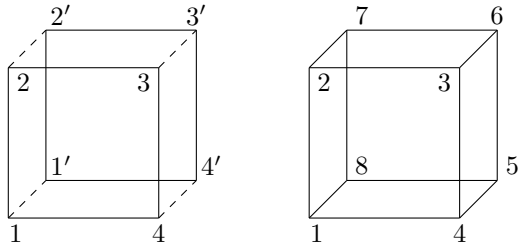

After some historical background, there is a section summing up many of the important results,
several of which are unpublished, pertaining to \(\ccw(Q_n)\).  The fourth section presents the
results from Ching J.\ Guu's Ph.D.\ dissertation \cite{Guu}, which are used in the new results of the
fifth section.

\section{Historical background}
\label{background}
Before 1996 \(\ccw(Q_3)\) was already known from using exhaustive searching by computer.  In
\cite{James} Beatrice James found an alternate method of determining \(\ccw(Q_n)\) and applied it to
\(Q_3\) and \(Q_4\).  Her method is extendable to higher dimension cubes, however, the number of
cases blows up.  In 2001 Ryan Aschenbrenner used another method \cite{Aschenbrenner} to find
\(\ccw(Q_5)\).  His method is also extendable, but has similar problems.  Currently, Candi Castillo
\cite{Castillo} is using Aschenbrenner's method to prove \(\ccw(Q_6)\).  So far as has been tested,
the CT conjecture has held.

During this 1996--present period other advances have been made with variations on the \(\ccw(Q_n)\)
problem.  In 1997 Ching J.\ Guu claimed in \cite{Guu} that \(\cwl(Q_n)\) is minimized with the
Graycode numbering.  In 2000 Bezrukov et al.\ \cite{BCHRS} published their proof that \(\lcw(Q_n)\)
is minimized with a lexicographic numbering.  In the proof they used an equivalent discrete
isoperimetric problem.  Also proved in that paper is the congestion of \(Q_n\), which is closely
connected to \(\lcw(Q_n)\).  In \cite{Harper} the Hale's numbering is shown to minimize
\(\lbw(Q_n)\).

\section{Important results}
\label{results}
\subsection{\(\ccw(Q_n)\)--Methods}
\label{methods}\({}\)\\
James's method \cite{James} is based on the fact that as far as cuts are concerned, there are only
two ways of looking at each disjoint \(Q_2\).  A \(Q_2\) can contribute one cut all the way around
the cycle or two cuts in a local region.  This method can be extended in two ways--simply increase
the cases as \(Q_n\) goes up in dimension or increase the size of the subcubes.  The second way
increases the number of cases because the number of ways of representing larger subcubes increases.
It appears that neither extension is suitable for proving \(\ccw(Q_n)\).

Aschenbrenner's method \cite{Aschenbrenner} involves a diameter that cuts the cycle in two pieces
and a pair of disjoint \(Q_{n-1}\)'s that are split by the diameter.  One can look at the size of
the split given a particular pair of \(Q_{n-1}\)'s with respect to a particular diameter.  The split
size is the number of vertices from each \(Q_{n-1}\) that are on one side of the diameter.  However,
more than one diameter may be considered.  In fact, the diameter is free to move around the cycle,
and the choice of \(Q_{n-1}\)'s is free as well, so long as they are complementary \(Q_{n-1}\)'s.

In looking at a single diameter, Aschenbrenner developed a useful notation to help find cutwidth.
His use of the notation is for \(Q_{n-1}\)'s though it can be extended to consider any pair of
complementary subgraphs:
\[   \left[
      \begin{array}{ccc}
  |A|      & |e_{AB}|            & |B|      \\
  \hline
  |e_{AC}| & |e_{AD}| + |e_{BC}| & |e_{BD}| \\
  \hline
  |C|      & |e_{CD}|            & |D|
      \end{array}
     \right].
\]

In the most general sense of the notation, vertex set \(A \cup C\) is that of one subgraph, and
vertex set \(B \cup D\) is the set for the other subgraph.  The first row indicates what is on one
side of the diameter.  The other side of the diameter is indicated by the third row.  The middle row
counts the number of edges that cross the diameter.  The number of edges that have one vertex in
\(X\) and the other vertex in \(Y\) is \(|e_{XY}|\).

In Aschenbrenner's paper it is mentioned that when there is a 5/11 split (or 11/5 split, depending
on which \(Q_4\) you count first and on which side of the cycle you count) the problem was easy for
\(Q_5\).  Generally, if there is a \(\frac{2}{3} / \frac{1}{3}\) split the problem is solved.  More
precisely, a \(\frac{2^n + (-1)^{n+1}}{3}\ /\ \frac{2^{n-1} + (-1)^n}{3}\) split is an
easy split for \(Q_n\).

\subsection{Proof for a \(\frac{2}{3} / \frac{1}{3}\) split}
\label{split}
\begin{thm}
\label{splitthm}
When there is at least a \(\frac{2^n + (-1)^{n+1}}{3}\ /\ \frac{2^{n-1} + (-1)^n}{3}\)
split, the largest cut is at least \(\frac{5 \cdot 2^{n-2} - 1}{3}\) when \(n\) is odd
or \(\frac{5 \cdot 2^{n-2} - 2}{3}\) when \(n\) is even.
\end{thm}

\begin{lemma}
\label{splitlem}
When there is a split greater than \(x/y\), an \(x/y\) split also occurs.
\end{lemma}
\begin{proof}[Proof of Lemma~\ref{splitlem}]
With no loss of generality we can assume \(x \geq y\).  Let \(k \geq 0\).  If an \(x+k / y-k\) split
exists, the diameter can be rotated one vertex pair at a time, \(2^{n-1}\) times.  At that point,
the diameter will be in the same position but oppositely oriented to its original position, giving a
\(y-k / x+k\) split.  With each rotation the left side of the split can increase by 1, decrease by
1, or stay the same.  Since \(y-k \leq x \leq x+k\), the left side of the split must have been \(x\)
at some point.  Thus, an \(x/y\) split exists.
\end{proof}

\begin{proof}[Proof of Theorem~\ref{splitthm}]
From the lemma, only a \(\frac{2^n + (-1)^{n+1}}{3}\ /\ \frac{2^{n-1} + (-1)^n}{3}\)
split has to be proven.  Using Aschenbrenner's notation, this split is written:
\[   \left[
      \begin{array}{ccc}
\frac{2^n + (-1)^{n+1}}{3}                             & \frac{2^{n-1} + (-1)^n}{3}  & \frac{2^{n-1} + (-1)^n}{3} \\
\hline
\theta_{n-1} \left( \frac{2^{n-1} + (-1)^n}{3} \right) & \frac{2^{n-1} + 2(-1)^{n+1}}{3} & \theta_{n-1} \left( \frac{2^{n-1} + (-1)^n}{3} \right) \\
\hline
\frac{2^{n-1} + (-1)^n}{3}                             & \frac{2^{n-1} + (-1)^n}{3}  & \frac{2^n + (-1)^{n+1}}{3}
      \end{array}
     \right].
\]

\noindent
Each vertex in each \(Q_{n-1}\) is connected by an edge to one vertex in the other \(Q_{n-1}\).  So
\(\frac{2^{n-1} + (-1)^n}{3}\) edges connecting \(Q_{n-1}\)'s is the maximum that can stay on each
side of the diameter.  This minimizes the number of edges between \(Q_{n-1}\)'s that go through the
diameter (See Figure~\ref{splitfig} and accompanying example).

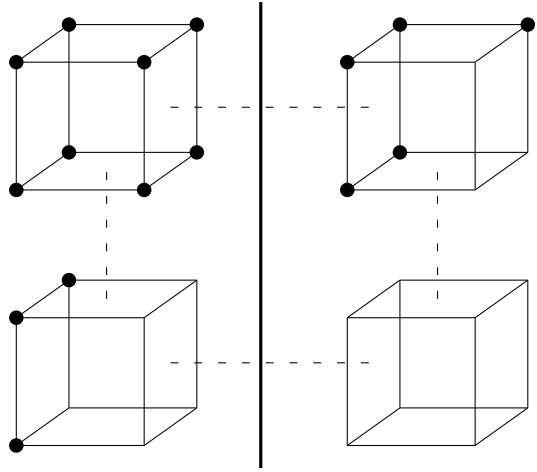
\begin{figure}[!h]
 \centering
  \begin{tikzpicture}
   \foreach \x in {0, 1.7, 4.4, 6.1}
    \foreach \y in {0, 1.7, 3.4, 5.1}
     \draw (\x,\y) -- +(0.7,0.5);
   \foreach \x in {0, 4.4}
    \foreach \y in {0, 1.7, 3.4, 5.1}
     {
     \draw (\x,\y) -- +(1.7,0);
     \draw (\x + 0.7,\y + 0.5) -- +(1.7,0);
     }
   \foreach \x in {0, 1.7, 4.4, 6.1}
    \foreach \y in {0, 3.4}
     {
     \draw (\x,\y) -- +(0,1.7);
     \draw (\x + 0.7,\y + 0.5) -- +(0,1.7);
     }
   \draw[loosely dashed] (1.2,1.95) -- +(0,1.7);
   \draw[loosely dashed] (5.6,1.95) -- +(0,1.7);
   \draw[loosely dashed] (2.05,1.1) -- +(2.7,0);
   \draw[loosely dashed] (2.05,4.5) -- +(2.7,0);
   \draw[very thick] (3.25,-0.3) -- +(0,6.2);
   \foreach \y in {0, 1.7, 3.4, 5.1}
    \node[vertex] at (0,\y) {};
   \foreach \x in {1.7, 4.4}
    \foreach \y in {3.4, 5.1}
     \node[vertex] at (\x,\y) {};
   \foreach \x in {0.7, 2.4, 5.1}
    \foreach \y in {3.9, 5.6}
     \node[vertex] at (\x,\y) {};
   \node[vertex] at (0.7,2.2) {};
   \node[vertex] at (6.8,5.6) {};
  \end{tikzpicture}
 \caption[\(Q_5\) as \(Q_3 \times Q_2\)]{\(Q_5\) as \(Q_3 \times Q_2\).  Solid vertices represent
          vertices on one side of the diameter.  \(Q_4\)'s are separated by the solid
          line.\label{splitfig}}
\end{figure}
In Aschenbrenner's notation Figure~\ref{splitfig} would be represented:
\[ \left[ \begin{array}{ccc}
    11 & 5 & 5  \\
\hline
    10 & 6 & 10 \\
\hline
     5 & 5 & 11
\end{array} \right].
\]

\noindent
Hence, the minimum number of edges that cross the diameter is
\[2 \cdot \theta_{n-1} \left( \frac{2^{n-1} + (-1)^n}{3} \right) + \frac{2^{n-1} + 2(-1)^{n+1}}{3},\]
and \(\ccw(Q_n)\) is at least half of this.  In \cite{BCHRS} a recursion is given for finding \(\theta_n(\ell)\):

\[
   \theta_n(\ell) = 
     \begin{cases}
        2 \ell + \theta_{n-2}(\ell)            & \text{if } 0 \leq \ell \leq 2^{n-2} \\
        2^{n-1} + \theta_{n-2}(\ell - 2^{n-2}) & \text{if } 2^{n-2} \leq \ell \leq 2^{n-1}
     \end{cases}
\]

Two cases arise--one when \(n\) is odd, and one when \(n\) is even.

\case[When \(n\) is odd]

\begin{align*}
\ccw(Q_n) & \geq \theta_{n-1} \left( \frac{2^{n-1} - 1}{3} \right) + \frac{2^{n-2} + 1}{3} \\
          & = \theta_{n-1} (1 + 2^2 + 2^4 + \dotsb + 2^{n-3}) + \frac{2^{n-2} + 1}{3} \\
          & = 2 \left( \frac{2^{n-1} - 1}{3} \right) + \frac{2^{n-2} + 1}{3} \\
          & = \frac{5 \cdot 2^{n-2} - 1}{3}.
\end{align*}
This is the same as \(\ccw(Q_n, \text{Graycode})\) for odd \(n\).

\case[When \(n\) is even]
\begin{align*}
\ccw(Q_n) & \geq \theta_{n-1} \left( \frac{2^{n-1} - 2}{3} + 1 \right) + \frac{2^{n-2} - 1}{3} \\
          & = \theta_{n-1} (1 + 2^1 + 2^3 + \dotsb + 2^{n-3}) + \frac{2^{n-2} - 1}{3} \\
          & = \frac{2^n - 1}{3} + \frac{2^{n-2} - 1}{3} \\
          & = \frac{5 \cdot 2^{n-2} - 2}{3}.
\end{align*}
This is the same as \(\ccw(Q_n, \text{Graycode})\) for even \(n\).
\end{proof}

\section{Cyclic wirelength of \(Q_n\)}
\label{cwl}
In Ching J.\ Guu's Ph.D.\ dissertation \cite{Guu}, a proof for \(\cwl(Q_n)\) is claimed.  Her claim
is that the Graycode numbering minimizes cyclic wirelength.  En route, she creates a derived network
to convert the problem into a discrete isoperimetric problem.  Then she defines
\define{\(\type(S)\)}, \(S \subseteq V_n\).  There are \(2n\) \(Q_{n-1}\)'s in \(Q_n\), which are
denoted \(H_i\).
\[  \type(S) = \min_{1 \leq i \leq 2n} |S \cap H_i|.\]

That \(\type(S)\) and the size of a split are related is not immediately obvious, so a more formal
(and slightly restrictive) definition for split size will be used in this section.  For \(S
\subseteq V_n\), the size of the split of \(S\) is,
\[  \Split(S) = \max_{1 \leq i \leq 2n} |S \cap H_i|.\]

\noindent
This definition gives only the bigger side of the split, but allows the splitting line to be a
nondiameter.  Now it should be fairly evident that \(\type(S) + \Split(S) = |S|\), and that
consequently, the splits used earlier are \(\type(S) / \Split(S)\) or \(\Split(S) / \type(S)\),
depending on the orientation of the splitting line.

\subsection{\emph{big} and \emph{small}}
\label{bigsmall}\({}\)\\
When \(|S| = 2^{n-1}\), \(0 \leq \type(S) \leq 2^{n-2}\).  Guu abbreviates \(\type(S)\) with \(t\),
and calls a set \define{\emph{big}} if \(t \geq 2^{n-3}\) and \define{\emph{small}} if \(t \leq
2^{n-3}\).  When a path in the derived network goes from a set to its complement, if all the sets in
the path are \emph{small}, the Graycode numbering is shown to minimize \(\cwl(Q_n)\).  When a set,
\(S'\), in the path is of \emph{big} type, Guu claims that \(\theta_n(S') \geq \frac{3}{4} \cdot
2^n\), which is large enough to not need any further consideration.  Her proof of this inequality,
however, contains at least one error.  It has not been determined how grave the error is.  Since the
approach may have some utility, the outline is included here.

First, \(f(x) = \frac{3}{4} - \frac{64}{7} (x - \frac{1}{2})^2\) is introduced, which has the
property that \(f(x-t) + f(x+t) + 2t \geq 2 f(x)\) when \(0 \leq t \leq \frac{7}{64}\).  Next, it is
demonstrated that when \(\frac{2^n}{24} \leq \type(S) \leq (\frac{1}{24} + \frac{7}{64}) 2^n\) and
\(|S| \leq 2^{n-1}\), \(\theta_n(S) \geq f(x) \cdot 2^n\).  The third and final step is that when
\(\type(S) \geq (\frac{1}{24} + \frac{7}{64}) 2^n\), \(\theta_n(S) \geq f(x) \cdot 2^n\).

An error occurs in the proof of the final step.  She takes \(S_1 \cup S_2 = S\), with \(S_1 \cap
S_2 \neq \{\}\).  Then she continues \(|S_1| + |S_2| = |S|\).

\section{New results}
\label{news}
\subsection{A new lower bound}
\label{lowerbd}\({}\)\\
Up to this point, the best lower bound known for \(\ccw(Q_n)\) was \(\ccw(Q_n) \geq \frac{1}{2}
\lcw(Q_n)\).  We also know an upper bound for \(\ccw(Q_n)\) (which is sharp at least up to \(n = 6\))
is \(\ccw(Q_n) \leq \floor*{ \frac{5}{8} \lcw(Q_n) }\).  (This formula is incorrect for the trivial
cases of \(n \leq 1\), when no cycles exist.)  For this section it will be assumed that Guu's
results are correct, and that the Graycode numbering minimizes the cyclic wirelength of \(Q_n\).
With this assumption, a larger lower bound can be calculated.  Before calculating this new lower
bound we first must know what the value of \(\cwl(Q_n)\) is.

\begin{claim*}
\(\cwl(Q_n) = 2^{2n-2} + 2^{2n-3} - 2^{n-1}\)
\end{claim*}

\begin{proof}[Proof of Claim]
This claim will be proven by summing the lengths of each wire in the Graycode numbering.  Starting
with the horizontal and vertical wires, there are 4 each of wires that have length \(1, \dotsc,
2^{n-1}-3, 2^{n-1}-1\).  There are \(2^2\) groups of edges along primary diagonals whose wires have
lengths \(1, \dotsc, 2^{n-2} - 3, 2^{n-2} - 1\).  There are \(2^3\) groups of edges along secondary
diagonals whose wires have lengths \(1, \dotsc, 2^{n-3} - 3, 2^{n-3} - 1\), etc.

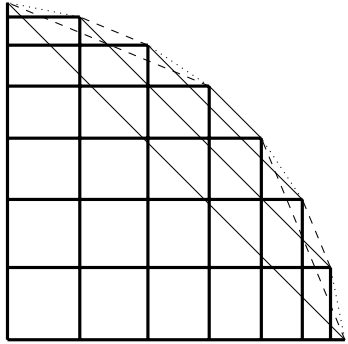
\begin{figure}[!h]
 \centering
  \begin{tikzpicture}
   \node[coordinate] (1) at (5.625:5) {};
   \node[coordinate] (2) at (16.875:5) {};
   \node[coordinate] (3) at (28.125:5) {};
   \node[coordinate] (4) at (39.375:5) {};
   \node[coordinate] (5) at (50.625:5) {};
   \node[coordinate] (6) at (61.875:5) {};
   \node[coordinate] (7) at (73.125:5) {};
   \node[coordinate] (8) at (84.375:5) {};
   \node[vertex] at (0,0) {};
   \draw (1) -- (8)
         (2) -- (7)
         (3) -- (6)
         (4) -- (5);
   \draw[dashed] (1) -- (4)
                 (2) -- (3)
                 (5) -- (8)
                 (6) -- (7);
   \draw[dotted] (1) -- (2)
                 (3) -- (4)
                 (5) -- (6)
                 (7) -- (8);
   \draw[very thick] (1) -- (1 -| 8)
                     (2) -- (2 -| 8)
                     (3) -- (3 -| 8)
                     (4) -- (4 -| 8)
                     (5) -- (5 -| 8)
                     (6) -- (6 -| 8)
                     (7) -- (7 -| 8)
                     (2) -- (2 |- 1)
                     (3) -- (3 |- 1)
                     (4) -- (4 |- 1)
                     (5) -- (5 |- 1)
                     (6) -- (6 |- 1)
                     (7) -- (7 |- 1)
                     (8) -- (8 |- 1);
  \end{tikzpicture}
 \caption[One quarter of a \(Q_5\) Graycode]{One quarter of a \(Q_5\) Graycode.\label{wirefig}}
\end{figure}
\vfill\eject
For example, Figure~\ref{wirefig} depicts a quarter of a \(Q_5\) Graycode.  The wires along the
primary diagonal are solid, the wires along the secondary diagonals are dashed, and the wires along
the tertiary diagonals are dotted.  The dot in the lower left corner is the location of the center
of the cycle.  Keep in mind, the horizontal and vertical wires have been truncated in the figure,
and that there are four times as many copies of each of the diagonal wires.

The grand sum of lengths of wires for \(Q_n\) is thus:

\begin{align*}
\cwl(Q_n) & = \begin{aligned}[t]
              & 4[(2^{n-1} - 1) + (2^{n-1} - 3) + \dotsb + 1] \\
              & + 2^2[(2^{n-2} - 1) + (2^{n-2} - 3) + \dotsb + 1] \\
              & + 2^3[(2^{n-3} - 1) + (2^{n-3} - 3) + \dotsb + 1] + \dotsb + 2^{n-1} \\
              \end{aligned}\\
          & = 4 (2^{2n-4}) + 2^2 (2^{2n-6}) + 2^3 (2^{2n-8}) + \dotsb + 2^{n-1} \\
          & = 2^{2n-2} + 2^{2n-4} + 2^{2n-5} + \dotsb + 2^{n-1} \\
          & = 2^{2n-2} + 2^{2n-3} - 2^{n-1}. \qedhere
\end{align*}
\end{proof}
Therefore, \(\ccw(Q_n) \geq \frac{\cwl(Q_n)}{2^n} = 2^{n-2} + 2^{n-3} - \frac{1}{2}\).  Or, in terms
of \(\lcw(Q_n)\):  \(\ccw(Q_n) \geq \floor*{ \frac{9}{16} \lcw(Q_n) }\).

\subsection{Open conjectures}
\label{conjectures}\({}\)\\
Reverse engineering this process starting with a desired lower bound gives the following:
\[ \cwl(Q_n) > 
     \begin{cases}
        2^n \left( \frac{5 \cdot 2^{n-2} - 1}{3} - 1 \right) & \text{if \(n\) is odd} \\
        2^n \left( \frac{5 \cdot 2^{n-2} - 2}{3} - 1 \right) & \text{if \(n\) is even.}
     \end{cases}
\]

\noindent
It would be nice to find this as the lower bound to \(\cwl(Q_n)\) for when a \(\frac{2}{3} /
\frac{1}{3}\) split does not exist.  If it is a lower bound, the CT conjecture would be verified.
This can be achieved if Guu's method can be applied with \(f(x) = \frac{5}{6} - k(x -
\frac{1}{2})^2\) and \(\type(S) \geq \frac{1}{3} 2^{n-1}\) (as opposed to her \(f(x) = \frac{3}{4} -
\frac{64}{7} (x - \frac{1}{2})^2\) and \(\type(S) \geq 2^{n-3}\)).

\section{Conclusion}
\label{conclusion}
When a \(\frac{2}{3} / \frac{1}{3}\) split exists, the Graycode numbering optimizes \(\ccw(Q_n)\).
The Graycode also is likely the optimal numbering for \(\cwl(Q_n)\), though there remain holes in
the proof.  If the Graycode does optimize \(\cwl(Q_n)\), the cyclic cutwidth problem is ``cut in
half.''  A stronger version of the cyclic wirelength problem could solve the cyclic cutwidth problem.

\section{Acknowledgments}
I would like to thank Dr.\ J.\ D.\ Chavez and Dr.\ R.\ Trapp for making this project an enjoyable
and productive experience.  This work was completed during the 2003 Research Experiences for
Undergraduates (REU) in Mathematics at California State University San Bernardino (CSUSB).  It was
sponsored jointly by CSUSB and NSF-REU Grant number DMS-0139426.

\bibliographystyle{plain}

\end{document}